\documentclass[11pt]{amsart}   	
\usepackage[margin=1in]{geometry}
\usepackage{amsmath,amssymb, amsfonts,amsthm,stackrel}
\usepackage[hide links]{hyperref}
\usepackage{mathtools}
\usepackage{tikz-cd}
\usetikzlibrary{decorations.markings}
\usepackage{makecell}
\setcellgapes{4pt}
\usepackage{multicol}
\usepackage[capitalise]{cleveref}
\usepackage[shortlabels]{enumitem}

\newtheorem{theorem}{Theorem}[section]
\newtheorem{lemma}[theorem]{Lemma}
\newtheorem{corollary}[theorem]{Corollary}
\newtheorem{proposition}[theorem]{Proposition}
\newtheorem{definition-proposition}[theorem]{Definition-Proposition}

\theoremstyle{definition}
\newtheorem{definition}[theorem]{Definition}

\newtheorem{remark}[theorem]{Remark}

\DeclareMathOperator{\add}{\mathsf{add}}
\DeclareMathOperator{\Filt}{\mathsf{Filt}}
\DeclareMathOperator{\Gen}{\mathsf{Gen}}

\DeclareMathOperator{\FiltGen}{\mathsf{T}}
\newcommand{\J}{\mathsf J}
\newcommand{\T}{\mathcal T}
\newcommand{\F}{\mathcal F}
\newcommand{\W}{\mathcal W}
\newcommand{\X}{\mathcal X}
\newcommand{\Y}{\mathcal Y}

\DeclareMathOperator{\mods}{\mathsf{mod}}

\DeclareMathOperator{\Hom}{\mathrm{Hom}}
\DeclareMathOperator{\Ext}{\mathrm{Ext}}
\DeclareMathOperator{\ind}{\mathrm{ind}}

\renewcommand{\P}{\mathcal{P}}
\DeclareMathOperator{\Ps}{\mathcal{P}_\mathrm{s}}
\DeclareMathOperator{\Pns}{\mathcal{P}_\mathrm{ns}}
\newcommand{\Pmod}{P}
\newcommand{\Pmods}{P_\mathrm{s}}
\newcommand{\Pmodns}{P_\mathrm{ns}}

\newcommand{\Ebm}{\mathsf E}
\newcommand{\Fbm}{\mathsf{E}^{-1}}

\newcommand{\V}{\mathcal V}

\newcommand{\SC}{\mathcal S}

\newcommand{\tes}{\tau\text{-}\mathrm{es}}
\newcommand{\tf}{\mathrm{tf}}

\usepackage{soul}

\title{Transitivity of mutation of $\tau$-exceptional sequences in the $\tau$-tilting finite case}
\author{Aslak B. Buan}
\address{Department of Mathematical Sciences, Norwegian University of Science and Technology (NTNU), 7491 Trondheim, NORWAY}
\email{aslak.buan@ntnu.no}
\author{Eric J. Hanson}
\address{Department of Mathematics, North Carolina State University, Raleigh, NC 27695, USA}
\email{ejhanso3@ncsu.edu}
\author{Bethany R. Marsh}
\address{School of Mathematics, University of Leeds, Leeds LS2 9JT, U.K.}
\email{B.R.Marsh@leeds.ac.uk}
\thanks{ABB was supported by 
 grant number FRINAT 301375 from the Norwegian Research Council. EJH is supported by an AMS-Simons travel grant. This work was supported by the Engineering and Physical Sciences Research Council [grant number EP/W007509/1].
 ABB would like to thank BRM and the School of Mathematics at the University of Leeds for their warm hospitality. EJH and BRM would also like to thank ABB and the Department of Mathematical Sciences at NTNU for their warm hopitality.}


\begin{document}
\begin{abstract}
We prove that mutation of complete $\tau$-exceptional sequences is transitive for $\tau$-tilting finite algebras.
\end{abstract}
\maketitle

\vspace{-11pt}

\section*{Introduction}

Inspired by $\tau$-tilting theory 
\cite{air} and, in particular,
$\tau$-tilting reduction \cite{jasso},
$\tau$-exceptional sequences were introduced in \cite{bm}. 
These are certain finite sequences of indecomposable objects in the category of finite-dimensional modules over an arbitrary fixed finite-dimensional algebra. The sequences have  
properties similar to classical exceptional sequences over hereditary algebras
\cite{cb, rin}. In particular, the rank  of the algebra (that is, the number of isomorphism classes of simple modules) gives an upper bound of the length of a $\tau$-exceptional sequence. 

Exceptional sequences were first studied in a triangulated setting, where they come with a mutation operation that induces a braid group action on the set of such sequences; see \cite{bon, goru}. This specializes to an operation on exceptional sequences in the module category of a hereditary 
(finite-dimensional) algebra. Crawley-Boevey \cite{cb} and Ringel \cite{rin} proved that this mutation operation is transitive in the hereditary case.

In \cite{bhm}, a mutation operation on $\tau$-exceptional sequences was introduced and proved to 
generalize the hereditary case. But, for arbitrary algebras, the operation does not respect the braid relations. Furthermore, it is not, in general, transitive on complete $\tau$-exceptional sequences.
However, it follows from \cite[Thm.~0.6]{bhm} that the mutation operation is transitive for algebras that are $\tau$-tilting finite and have
rank two. (Recall from \cite{dij} that an algebra
is $\tau$-tilting finite if there
are only finitely many isomorphism classes of basic support $\tau$-rigid modules.) 
Furthermore, in \cite[Thm.~1.2]{bkt} the authors obtained the same transitivity result for Nakayama algebras. 

In this paper, we prove that
mutation of complete $\tau$-exceptional sequences is {\em always transitive} in the 
$\tau$-tilting finite case.
We prove our main result in Section \ref{sec2}, after giving the necessary notation and background in Section \ref{sec1}.

\section{Background}\label{sec1}
We consider algebras which are basic and finite-dimensional over a field $\Bbbk$. We denote the category of finite-dimensional (left) $\Lambda$-modules over an algebra $\Lambda$ by $\mods \Lambda$. We only consider basic modules and let $\delta(M)$ denote the number of indecomposable direct summands of a module $M$. We let 
$n$ denote the rank of $\Lambda$. We consider $\Lambda$-modules up to isomorphism, and all subcategories considered are full and closed under isomorphisms.

Let $\tau$ denote the Auslander-Reiten translate in $\mods \Lambda$. Following \cite{air}, a module $M$ is said to be
{\em $\tau$-rigid} if $\Hom(M, \tau M) = 0$.
If, in addition, $\delta(M) = n$, the module $M$ is said to be {\em $\tau$-tilting}. We assume that $\Lambda$ is $\tau$-tilting finite throughout.

We will consider pairs of objects $(X,Y)$ in $\mods \Lambda$. In particular, a {\em support $\tau$-rigid object}~\cite[Defn.~0.3]{air} is a pair of modules 
$(M,P)$, where $M$ is $\tau$-rigid and $P$ is projective with $\Hom(P,M)= 0$. When convenient, we 
consider instead the corresponding object $X \oplus Y[1]$ in the 
subcategory $\mods \Lambda \amalg \mods \Lambda[1]$ of the bounded derived category of $\mods \Lambda$.

Given a $\Lambda$-module $M$, we have the left and right perpendicular subcategories $$M^\perp = \{X \in \mods \Lambda \mid \Hom(M,X) = 0\} \text{ and } {}^\perp M = \{Y \in \mods \Lambda \mid \Hom(Y,M) = 0\}.$$

\subsection{Torsion classes and wide subcategories}
\label{sec:torsionwide}
For a subcategory $\mathcal{X}$ of $\mods \Lambda$,
we let $\Gen \mathcal{X}$ 
denote the closure of $\mathcal{X}$ under factor objects and we let $\Filt \mathcal{X}$ denote the closure under filtrations. We denote by $\ind \mathcal{X}$ the set of indecomposable objects in $\mathcal{X}$ (up to isomorphism).
A module $Q$ in $\mathcal{X}$ is called {\em $\Ext$-projective} in $\mathcal{X}$ if $\Ext^1(Q,\mathcal{X})=0$. An indecomposable $\Ext$-projective $Q$ which is not 
in $\Gen (\mathcal{X} \setminus {\add Q})$, is called 
{\em split projective}, and otherwise it is {\em non-split projective}. The set of 
$\Ext$-projectives (respectively split projectives, non-split projectives) in $\mathcal{X}$ is denoted by $\P(\mathcal{X})$ (respectively $\Ps(\mathcal{X}), \Pns(\mathcal{X})$).
If $\mathcal{X}$ has a finite number of indecomposable 
$\Ext$-projectives, we denote by $\Pmod(\mathcal{X})$
(respectively $\Pmods(\mathcal{X}), \Pmodns(\mathcal{X})$),
the direct sum of one copy of each indecomposable 
$\Ext$-projective (respectively, split projective, non-split projective) module.
We let $\P(\Lambda)$ denote the subcategory of projective $\Lambda$-modules.

For a module $M$, we let $\Gen M = \Gen (\add M)$, where $\add M$ denotes the additive closure of $M$. A module $M$ is
{\em gen-minimal} if, for any proper direct summand $M'$ 
of $M$, we have $\Gen M' \subsetneq \Gen M$.

{\em Torsion classes} are subcategories of $\mods \Lambda$ closed under factor modules and extensions, and  
for any subcategory $\mathcal{S}$ of $\mods \Lambda$, we have that
$\mathsf{T}(\mathcal{S}) = \Filt (\Gen \mathcal{S})$
is the smallest torsion class
containing $\mathcal{S}$ (see \cite{ms} or \cite{bhm}.)
Dually, torsion-free classes are subcategories of $\mods \Lambda$ closed under submodules and extensions.
For a $\tau$-rigid module $M$, the subcategory $\Gen M$ is a torsion class by~\cite[Thm.~5.10]{as81}.
A {\em torsion pair} in $\mods \Lambda$ is a pair of subcategories 
$(\T,\F)$ such that $\T = {^\perp \F}$ and $\T^{\perp} = \F$.
If $(\T,\F)$ is a torsion pair, then $\T$ is a torsion class and $\F$ is a  torsion-free class.
Moreover, if $\T$ is any torsion class in $\mods \Lambda$, then $(\T,\T^{\perp})$ is a torsion pair, and all torsion pairs are of this form.

Let $(\T,\T^{\perp})$ be a torsion pair and let $M$ be a module. There is then a unique
exact sequence
$
0 \to t_{\T}M \to M \to f_{\T} M \to 0
$
with $t_{\T}M \in \T$ and $f_{\T}M \in {\T}^\perp$,  and this induces functors
$t_\T:\mods \Lambda\to \T$ and $f_\T:\mods \Lambda\to \T^{\perp}$ (see, for example,~\cite[VI.1]{ASS06}).
If the torsion class $\T$ is equal to $\Gen M$ for some module $M$, we denote the functor $f_{\T}$ simply by $f_M$.

A subcategory $\mathcal{X}$ of $\mods \Lambda$ is called
{\em wide} if it closed under kernels, cokernels and extensions.
For a support $\tau$-rigid object $(M,P)$ we consider the
{\em $\tau$-perpendicular} category 
$\J(M,P) = M^{\perp} \cap
{}^\perp \tau M \cap P^{\perp}$, which is wide by \cite[Thm.\ 3.28]{dirrt} (see also~\cite[Cor.\ 3.22]{bst}). 
We let $\J(M) := \J(M,0)$ for a 
$\tau$-rigid module $M$. Note that {\em Serre subcategories} are exactly $\tau$-perpendicular subcategories of the form
$\J(P,0) (= \J(0,P))$ for $P$ a projective $\Lambda$-module.

There are precise interrelations between the sets of wide subcategories, torsion classes and 
gen-minimal $\tau$-rigid modules. We will, in particular, be dealing with the $\tau$-tilting finite case, where we have the following strong results.

\begin{theorem}\label{bijections}
\cite{air,bh,bhm,dij,dirrt,jasso,ms}
Let $\Lambda$ be a $\tau$-tilting finite algebra.
\begin{itemize}
    \item[(a)] The indicated maps define inverse bijections
$$
\begin{tikzcd}[column sep=1.1cm, row sep=1cm]
\{\text{Gen-minimal $\tau$-rigid modules}\}
\ar[rr, shift right=.4em,swap, "M \mapsto \Gen M"] 
\ar[ddr, shift right=.4em, swap,"M \mapsto 
{\J(\Pmodns(\Gen M),P_{\Gen M})}"] & & 
\{\text{Torsion classes}\} \ar[ll,
shift right=.4em, swap,"\T \mapsto \Pmods(\T)"] 
\ar[ddl, shift right=.4em, swap,
"\T \mapsto {\J(\Pmodns(\T),P_{\T})}"] \\
& &  \\
& \{\text{Wide subcategories} 
\ar[uur,shift right=.4em,swap,"\W \mapsto \mathsf{T}(\W)= \Filt(\Gen \W)"] \ar[luu, shift right=.4em, swap,"\W \mapsto \Pmods(\mathsf{T}(\W))"]\}
\end{tikzcd}$$
 where for a torsion-class $\T$, we let 
$P_{\T}$ denote the direct sum of one copy of each  indecomposable projective $Q$ with $\Hom(Q, \T) = 0$.
\item[(b)]
All wide subcategories are
$\tau$-perpendicular categories, and hence equivalent to module categories of $\tau$-tilting finite algebras. Furthermore, the rank of a $\tau$-perpendicular category $\J(M,P)$ (with $(M,P)$ a support $\tau$-rigid object) is 
given by $n-\delta(M) -\delta(P)$.

\end{itemize}
\end{theorem}

\begin{proof}
We briefly explain how to extract these results from the literature.

Since $\Lambda$ is $\tau$-tilting finite, there are finitely many torsion classes, all of which are functorially finite, by~\cite[Thm.~1.2 and Cor.~2.9]{dij}. One can then obtain the horizontal bijections in (a)
 from~\cite[Sect.\ 2.3]{air}; see also~\cite[Thm.\ 1.6]{bhm}. The bijection from torsion classes to wide subcategories is then from \cite[Cor.~3.11]{ms}. The formula for the inverse of this bijection is from~\cite[Lem.~4.3]{bh}. Finally, the maps on the left hand side of the diagram are compositions of these bijections, or see~\cite[Thm.~4.18]{dirrt}.

 The fact that all wide subcategories are $\tau$-perpendicular categories was proved in \cite[Thm.~4.16]{dirrt}, and is also contained in (a). The rest of statement (b) is from \cite[Thm.~4.12]{dirrt} and \cite[Thm.~1.4 and Thm.~1.5]{jasso}.
\end{proof}

Let $M$ be a $\tau$-rigid module. Then the 
subcategory ${}^\perp{\tau M}$ is a functorially finite torsion class,
and the module $\Pmod({}^\perp{\tau M})$ is 
a $\tau$-tilting module, called the {\em Bongartz-completion} of $M$
(see \cite[Thm.\ 2.10]{air}). Similarly, $\Gen M$ is a functorially finite torsion class and $\Pmod(\Gen M)$ is called
the {\em co-Bongartz-completion} of $M$ ~\cite[Sect.\ 2.3]{air} (see also~\cite[Thm.\ 1.6]{bm}). It is a $\tau$-tilting module in the Serre subcategory $\mods \Lambda / \langle e \rangle $, where
$e$ is the idempotent such that $\Lambda e = P_{\Gen M}$.

We conclude this subsection with the following.

\begin{lemma}\label{lem:rigid_unique}
    Let $M$ and $N$ be $\tau$-rigid modules. Then any two of the following implies the third.
    \begin{enumerate}
        \item $\Gen M = \Gen N$.
        \item ${}^\perp \tau M = {}^\perp \tau N$.
        \item $\J(M) = \J(N)$.
    \end{enumerate}
    Moreover, these three conditions hold if and only if $M = N$.
\end{lemma}

\begin{proof}
    It is clear that (1) and (2) together imply (3). Let $(\Gen M) * \J(M)$ denote all modules that appear as midterms of short exact sequences with the left term in
    $\Gen M$ and the right term in $\J(M)$.
    By \cite[Thm.~3.12]{jasso}, we have
    $${}^\perp \tau M = (\Gen M) * \J(M),$$
and likewise for $N$. This shows that (1) and (3) together imply (2). Finally, intersecting both sides of the equation above with ${}^\perp \J(M)$ yields
    $${}^\perp \tau M \cap {}^\perp \J(M) = \Gen M,$$
    and likewise for $N$.
    This shows that (2) and (3) together imply (1).

    For the moreover part, suppose that (1) and (2) hold. Now (1) implies that $N \in \P(\Gen M)$, and so $N$ is a direct summand of the co-Bongartz completion $M \oplus C_M$ of $M$. Similarly (2) implies that $N \in \P({}^\perp \tau M)$, and so $N$ is a direct summand of the Bongartz completion $M \oplus B_M$ of $M$. Now the modules $B_M$ and $C_M$ have no direct summands in common by \cite[Lem.~1.7]{bhm}, so together these facts imply that $N$ is a direct summand of $M$. Reversing the roles of $M$ and $N$, we see that also $M$ is a direct summand of $N$. We conclude that $M = N$.
\end{proof}

\subsection{$\tau$-exceptional sequences}

Following \cite{bm}, a sequence $(X_k,X_{k+1},\dots,X_n)$ of indecomposable $\Lambda$-modules is called a {\em $\tau$-exceptional sequence} in $\mods\Lambda$ if
    \begin{itemize}
        \item[-] $X_n$ is $\tau$-rigid in $\mods \Lambda$, and
        \item[-] if $k <n$, then $(X_k,\dots,X_{n-1})$ is a $\tau$-exceptional sequence in $\J(X_n)$.
    \end{itemize}

\begin{remark}
    Note that, by \cref{bijections}(b), the wide subcategory $\J(X_n)$ is equivalent to a module category, so there is an Auslander-Reiten translate $\tau_{\J(X_n)}$, which in general is {\em different from} $\tau$ in $\mods \Lambda$. Saying that a module $M \in \J(X_n)$ is $\tau$-rigid in $\J(X_n)$ thus means that $M$ is $\tau_{\J(X_n)}$-rigid; that is, that $\Hom(M,\tau_{\J(X_n)}M) = 0$. Similarly, saying that a pair $(M,P)$ (identified with the object $M \oplus P[1]$) is support $\tau$-rigid in $\J(X_n)$ means that $M$ is $\tau_{\J(X_n)}$-rigid, $P \in \P(\J(X_n))$, and $\Hom(P,M) = 0$.
\end{remark}

Let $M$ be a $\Lambda$-module with $\delta(M) = t$.
A sequence $(M_1, M_2, \dots, M_t)$ of indecomposable modules is an \emph{ordering} of $M$ if
$\oplus_{i=1}^t M_i = M$, and a 
{\em TF-ordering}~\cite[Defn.\ 3.1]{mt} if $M_i \not \in \Gen(\oplus_{j>i} M_j)$ for all $i$.
By a {TF-ordered} $\tau$-rigid module, we will mean
a $\tau$-rigid module with a fixed TF-ordering. 
So, for example, the
module $_{\Lambda}{\Lambda}$ 
gives rise to $n!$ different TF-ordered $\tau$-rigid modules. Since $\tau$-rigid modules have rank at most $n$, we typically denote a TF-ordered $\tau$-rigid module by $\Y = (Y_k,\ldots,Y_n)$. In this case, we write $\bigoplus \Y = Y_k \oplus \cdots \oplus Y_n$.

TF-ordered $\tau$-rigid modules and $\tau$-exceptional sequences are related by the following.
    
\begin{theorem}\cite[Thm.\ 5.1]{mt}
Let $\Y= (Y_k,Y_{k+1}, \dots,Y_n)$ be a 
TF-ordered $\tau$-rigid module, and for all $k \leq i \leq n$,
let $X_i = f_{\oplus N_{j>i}}Y_i$. 
Then $\X =(X_k,X_{k+1}, \dots,X_n)$ is a $\tau$-exceptional sequence. Moreover, the map $\omega: \Y \mapsto \X$ is a bijection from the set of TF-ordered $\tau$-rigid modules with $\delta(\bigoplus \Y) =n -k+1$ to the set of $\tau$-exceptional sequences of length $n-k+1$.
\end{theorem}

The following map, extending Jasso's reduction map \cite[Thm. 3.16]{jasso}, is crucial for our results.

\begin{proposition}
    \cite[Sect. 5]{bm}\label{def:Emap} For any support $\tau$-rigid object $T=(M,P)$, there is  a bijection
$$\{\text{indecomposable support $\tau$-rigid objects $X$ in $\mods \Lambda$ with $T \oplus X$ support $\tau$-rigid} \}$$
$$\downarrow \Ebm_T$$
$$\{\text{indecomposable support $\tau$-rigid objects $Y$ in $\J(T)$}\}.$$
If $T=(M,0)$ we let $\Ebm_M := \Ebm_{(M,0)}$, and if $X$ is a
module such that $(X,0)$ is in the domain of $\Ebm_T$, we let
$\Ebm_T(X) := \Ebm_T(X,0)$. 

With this notation and with $X$ a module in the domain of $\Ebm_M$, we have 
\begin{itemize}
\item[(a)] $\Ebm_{M} (X)= f_M(X)$ for
$X \not \in \Gen M$, and
\item[(b)] $\Ebm_{M} (X) = Q[1]$ for some module $Q$ which is projective in $\J(M)$, otherwise.
\end{itemize}
Furthermore, extending $\Ebm_T$ additively, we have that 
$\Ebm_T(X)$ is support $\tau$-rigid in $\J(T)$, for any $X$ such
that $T \oplus X$ is support $\tau$-rigid in $\mods \Lambda$.
\end{proposition}

\begin{remark} \label{rem:allmodule}
For some of the notions and constructions defined above, we will also use relative versions 
with respect to a $\tau$-perpendicular, and hence wide, subcategory $\W$. 
Note that, by \cite[Thm.\ 3.8]{jasso}, $\tau$-perpendicular categories
of $\mods \Lambda$ are again
equivalent to module categories. Also note that in the $\tau$-tilting finite case, we have,
from \cref{bijections}(b), that
all wide subcategories of $\mods \Lambda$ are again
equivalent to module categories of $\tau$-tilting finite algebras.

When using the relative versions, the wide subcategory $\W$ will appear in the sub- or superscript. For example, we will consider:
\setlength{\columnsep}{-10cm}
\begin{multicols}{3}
    \begin{itemize}
        \item $\J_{\W}(-)$
        \item $\Ebm_{(-)}^{\W}(-)$
        \item $\Gen_{\W} (-)$
        \item $\omega_{\W} (-)$   
        \item $f_{\T}^{\W}(-)$
        \item $f_M^{\W}(-)$
    \end{itemize}
    \end{multicols}
\noindent Note that $\Gen_{\W} M = \Gen M \cap \W$.
\end{remark}   

\begin{definition}\label{def:defs}
    \begin{enumerate}
    \item For a $\tau$-exceptional sequence 
    $(X_k, X_{k+1}, \dots, X_n)$, we define recursively
$$\J(X_k, X_{k+1}, \dots, X_n)= \J_{\J(X_{k+1}, \dots, X_n)}(X_k),$$
    for $k<n$.
        \item For a $\tau$-perpendicular subcategory $\W$, we denote the set
        of $\tau$-exceptional sequences (resp. TF-ordered $\tau$-rigid modules) $\mathcal{X} = (X_k,\ldots,X_n)$ such that $\J(\mathcal{X}) = \W$ (resp. $\J(\bigoplus\mathcal{X}) = \W$) by $\tes(\W)$ (resp. $\tf(\W)$)
    \end{enumerate}
\end{definition}

We conclude this subsection by recalling and establishing the properties of the maps $\J$, $\Ebm_M$, $\omega$, and $f_M$ that we need.

The following important property of the $E$-map 
was proved in the $\tau$-tilting finite case in \cite[Thm. 5.9]{bm2}, and then generalized in 
\cite{bh}.

\begin{lemma}\cite[Thm. 6.12]{bh}\label{lem:E_map_sum}
    Let $X \oplus Y \oplus Z$ be a support $\tau$-rigid object. Then
    $$\Ebm_{Y\oplus Z}(X) = \Ebm_{\Ebm_Z(Y)}^{\J(Z)}(\Ebm_Z(X)).$$
\end{lemma}

We will also need the following observation.

\begin{lemma}\label{lem:gen_pass_down}
    Let $X \oplus Y \oplus Z$ be a $\tau$-rigid module with $X$ and $Y$ indecomposable. Suppose that $Y \notin \Gen Z$, that $X \notin \Gen Z$, and that $X \in \Gen(Y\oplus Z)$. Then $\Ebm_Z X = f_Z X \in \Gen(f_Z Y) = \Gen(\Ebm_Z Y)$.
\end{lemma}

\begin{proof}
    We have $f_Z X = \Ebm_Z X$ and $f_Z Y = \Ebm_Z Y$ by Lemma~\ref{def:Emap}. Since $f_Z$ is right exact and $f_Z Z = 0$, the result follows from applying $f_Z$ to an exact sequence of the form
    $(Y\oplus Z)^r\rightarrow X\rightarrow 0$.
\end{proof}

\begin{lemma}\cite[Lem.~4.9]{bm}\label{lem:proj_bij}
    Let $M$ be a $\tau$-rigid module. Then $f_M$ induces a bijection
    $$\mathrm{ind}(\P({}^\perp \tau M))\setminus \mathrm{add}(M) \rightarrow \ind(\P(\J(M))).$$
\end{lemma}

\begin{proposition}\cite[Cor.~3.6]{bhm}\label{prop:uniqueness}
    Let $(X,Y)$ and $(X',Y')$ be $\tau$-exceptional sequences with $\J(X,Y) = \J(X',Y')$. Then $X = X'$ if and only if $Y = Y'$.
\end{proposition}

\begin{lemma}\label{lem:J_sum}\cite[Thm.~6.4]{bh}
    Let $X \oplus Y$ be a support $\tau$-rigid object. Then $\J(X\oplus Y) = \J_{\J(Y)}(\Ebm_Y(X))$.
\end{lemma}

We then have the following direct consequence of Lemmas~\ref{lem:J_sum} and~\ref{lem:E_map_sum}. 

\begin{proposition}\label{prop:same_J}
    Let $\mathcal{X}$ be a $\tau$-exceptional sequence. Then $\J(\mathcal{X}) = \J(\bigoplus \omega^{-1}(\mathcal{X}))$. In particular, for $\W$ a $\tau$-perpendicular subcategory, $\omega$ restricts to a bijection $\tf(\W) \rightarrow \tes(\W)$.
\end{proposition}

As a consequence of Lemmas~\ref{lem:E_map_sum} and~\ref{lem:J_sum}, we obtain the following. See also \cite[Thm.~5.1]{mt}.

\begin{proposition}\label{prop:TF_tes_bijection}
    Let $(Y_k,\ldots,Y_n)$ be a TF-ordered $\tau$-rigid module and let $\omega(Y_k,\ldots,Y_n) = (X_k,\ldots,X_n)$. For $k \leq i \leq n+1$, denote $M_i = \bigoplus_{j = i+1}^n Y_j$. 
    Then, for $k \leq i \leq n$,  we have   $$X_i = \Ebm_{M_{i}} (Y_i) =
     \Ebm_{X_{i+1}}^{\J(M_{i+1})}(\Ebm_{M_{i+1}}Y_i).
    $$
\end{proposition}

\begin{corollary}\label{cor:uniqueness}
    Let $(X,Y)$ and $(Z,Y)$ be TF-ordered $\tau$-rigid modules with $\J(X\oplus Y) = \J(Z\oplus Y)$. Then $X = Z$.
\end{corollary}
\begin{proof}
    By Proposition~\ref{prop:TF_tes_bijection}, we get $\tau$-exceptional sequences $\omega(X,Y) = (\Ebm_Y X, Y)$ and $\omega(Z,Y) = (\Ebm_{Y}Z,Y)$. Proposition~\ref{prop:uniqueness} then implies that $\Ebm_{Y}X = \Ebm_{Y}Z$. Since $\Ebm_Y$ is a bijection, this shows that $X = Z$.
\end{proof}
The following description of the torsion class
generated by a $\tau$-exceptional sequence will be useful.

\begin{lemma}\label{lem:same_torsion}
    Let $\mathcal{X}$ be a $\tau$-exceptional sequence. Then $\Gen(\omega^{-1}(\mathcal{X})) = \FiltGen(\mathcal{X})$.
\end{lemma}

\begin{proof}
    Write $\omega^{-1}(\mathcal{X}) = (Y_k,\ldots,Y_n)$, $\mathcal{X} = (X_k,\ldots,X_n)$, and for $k \leq i \leq n+1$ denote $M_i = \bigoplus_{j = i+1}^n Y_j$. For $k \leq i \leq n$, Proposition~\ref{prop:TF_tes_bijection} and \cref{def:Emap} say that there is a short exact sequence
    $$0 \rightarrow t_{M_i} Y_i \rightarrow Y_i \rightarrow X_i \rightarrow 0$$
    with $X_i = f_{M_i} Y_i \in M_i^\perp$. We see immediately that $X_i \in \Gen Y_i$, and so $\FiltGen(\mathcal{X}) \subseteq \Gen(\omega^{-1}(\mathcal{X}))$. For the reverse inclusion we use reverse induction on $k$. For $k = n$, we have $Y_n = X_n$. For $k < n$, we have that $t_{M_k}Y_k \in \FiltGen(\mathcal{X})$ by the induction hypothesis, so the short exact sequence implies that also $Y_k \in \FiltGen(\mathcal{X})$.
\end{proof}

\subsection{Mutation of $\tau$-exceptional sequences}

A $\tau$-exceptional sequence of length two will be referred 
to as a {\em $\tau$-exceptional pair}.
Following \cite{bhm},
a $\tau$-exceptional pair $(B,C)$ is {\em left regular} if
$C$ is projective or $C \not \in \P({}^\perp{\tau \Fbm_C(B)})$. Otherwise, it is called \emph{left irregular}.
A $\tau$-exceptional pair $(X,Y)$ is {\em right regular} 
if $\Fbm_Y(X) \in \P({}^\perp{\tau Y})$ or $Y \not \in \Gen \Fbm_Y(X)$. Otherwise it is {\em right irregular}.
 
For $U$ indecomposable in $\mods \Lambda \amalg \mods \Lambda[1]$, we denote
$$|U| = \begin{cases} U, & \text{if $U \in \mods\Lambda$}; \\ U[-1], & \text{if $U \in \mods\Lambda[1]$.}\end{cases}$$

For a $\tau$-exceptional pair $(M,N)$, we let
\begin{equation*}
\begin{split}
    N_+ &:= \begin{cases} N[1], & \text{if $N$ is projective;}\\ N, & \text{otherwise;} \end{cases}\\
    \end{split}
    \hspace{20mm}
    \begin{split}
    M^+ &:= \begin{cases} M[1], & \text{if $M$ is projective in $\J(N)$;}\\ M, & \text{otherwise.} \end{cases}\\
    \end{split}
\end{equation*}
Furthermore, we let 
$M_{N\uparrow} := \Fbm_{N_+}(M)$ and 
        $M^+_{N\uparrow} := \Fbm_{N}(M^+)$
and note that we have $|N_+| = N$ and $|M^+| = M$.

In \cite[Sect.~4]{bhm}, we introduced partially-defined mutually inverse ``mutation operators'' $\varphi$ (left mutation) and $\psi$ (right mutation) on the set of $\tau$-exceptional pairs. For $\tau$-tilting finite algebras,  all $\tau$-exceptional pairs are (left and right) ``mutable''; that is, every $\tau$-exceptional pair is in the domain of both $\varphi$ and $\psi$. 

For a left regular $\tau$-exceptional pair $(B,C)$, we have
\begin{equation}\label{eqn:left}\varphi(B,C) = (\lvert \Ebm_{B_{C\uparrow}}(C_+) \rvert, B_{C\uparrow}).\end{equation}
For a right regular $\tau$-exceptional pair $(X,Y)$, we have
\begin{equation}\label{eqn:right}\psi(X,Y) = (\Ebm_{X^+_{Y\uparrow}}(Y), \lvert X^+_{Y\uparrow}\rvert).\end{equation}
The explicit formulas for $\varphi$ and $\psi$ in the left/right irregular cases will not be needed in this paper.

\begin{definition} \label{def:imutation}
Left mutation in a wide subcategory $\W$ is denoted 
by $\varphi^{\W}$, and
for a $\tau$-exceptional sequence $(X_k, X_{k+1}, \dots, X_n)$, and $i \geq k$,
we define the {\em left $i$-mutation} $\varphi_i$ by replacing 
the pair $(X_i, X_{i+1})$ with $\varphi^{\J(X_{i+2}, \dots, X_n)}(X_i, X_{i+1})$.
{\em Right $i$-mutation} $\psi_i$ is defined similarly.
\end{definition}

Mutation preserves $\tau$-perpendicular categories, in the following sense.

\begin{proposition}\label{prop:same_J_mutation}\cite[Cor.~5.3]{bhm}
    Let $\mathcal{X}$ = $(X_k,\ldots,X_n)$ be a $\tau$-exceptional sequence and consider $k \leq i < n$.
    Then $\J(\varphi_i(\mathcal{X})) = \J(\mathcal{X}) =
         \J(\psi_i(\mathcal{X}))$.
\end{proposition}

\subsection{Gen-minimal $\tau$-rigid modules}

For the proof of our main result, we will need some further results on
$\Gen$-minimal $\tau$-rigid modules, so we recall additional results from
\cite{bhm}, and point out some easy consequences.

\begin{lemma}\cite[Lem.~1.12]{bhm}\label{lem:proper_containment}
    Let $U \neq V$ be indecomposable $\tau$-rigid modules and suppose that $V \in \Gen U$. Then $\Gen V \subsetneq \Gen U$.
\end{lemma}

Recall that for any subcategory $\SC$, we have that ${}^\perp \SC$ is a torsion class.

\begin{theorem}\label{thm:gen_min_characterization}\cite[Thm.~2.14]{bhm}
    A $\tau$-rigid module $M$ is gen-minimal if and only if $M = \Pmods({}^\perp \J(M))$.
\end{theorem}

The following consequence is an important ingredient for the proof of our main result.

\begin{corollary}\label{cor:gen_min_unique}
    Let $\W$ be a $\tau$-perpendicular category and let $\mathcal{X}, \mathcal{X}' \in \tes(\W)$. If $\bigoplus \omega^{-1}(\mathcal{X})$ and $\bigoplus \omega^{-1}(\mathcal{X}')$
    are both gen-minimal, then $\bigoplus \omega^{-1}(\mathcal{X}) = \bigoplus \omega^{-1}(\mathcal{X}')$.
\end{corollary}
\begin{proof}
Since $\mathcal{X},\mathcal{X}'\in \tes{\W}$, we have
$\J(\X)=\J(\X')=\W$. Hence $\J(\bigoplus\omega^{-1}\X)=\J(\bigoplus\omega^{-1}\X')$ by Proposition~\ref{prop:same_J}.
This implies that $\Pmods({}^\perp \J(\bigoplus\omega^{-1}\X))=\Pmods({}^\perp \J(\bigoplus\omega^{-1}\X'))$. The result now follows from Theorem~\ref{thm:gen_min_characterization}.
\end{proof}

\section{Transitivity in the \texorpdfstring{$\tau$}{tau}-tilting finite case}\label{sec2}
In this section we prove our main result, transitivity of mutation of complete $\tau$-exceptional sequences.
Suppose throughout this section that $\Lambda$ is $\tau$-tilting finite. We freely use the facts
that, in this case, every wide subcategory is $\tau$-perpendicular (recalled in Theorem~\ref{bijections}), and that every $\tau$-exceptional sequence is left and right $i$-mutable for any $i$~\cite[Cor. 0.4]{bhm}.
 
We give a more general formulation, where
we first fix a wide subcategory $\W$ and then consider the transitivity of mutation for 
all $\tau$-exceptional sequences $\X$ with $\J(\X) =\W $.
For $\W = 0$, this gives transitivity of mutation for 
all complete $\tau$-exceptional sequences.

\begin{lemma}\label{lem:corank_2}
    Let $\W$ be a $\tau$-perpendicular subcategory of rank $n-2$. Then $\J(\Pmods({}^\perp \W)) = \W$ and there are indecomposable $U$ and $V$, such that
    $\Pmods({}^\perp\W) = U \oplus V$. Moreover $U \oplus V$ is gen-minimal.
\end{lemma}

\begin{proof}
    The fact that $\J(\Pmods({}^\perp \W)) = \W$ follows from \cite[Lem.~2.21]{bhm}. Note that the part of the proof of~\cite[Lem.~2.11]{bhm} showing this does not need the sincerity assumption made in the statement. The decomposition then follows from Theorem~\ref{bijections}(b) and the assumption that $\W$ has rank $n-2$. 
    The last statement follows from the definition of split projectivity.
\end{proof}

\begin{lemma}\label{lem:corank2_1}
    Let $\W$ be a wide category of rank $n-2$. As in Lemma~\ref{lem:corank_2}, decompose $\Pmods({}^\perp \W) = U \oplus V$ with $U$ and $V$ indecomposable.
    Then there exists $\ell \in \mathbb{Z}$ such that $\varphi^\ell(\omega(U,V)) = \omega(V,U)$.
\end{lemma}

\begin{proof}
    Note first that $\J(U\oplus V) = \W$, where $U \oplus V$ is gen-minimal by Lemma~\ref{lem:corank_2}. Thus, both $(U,V)$ and $(V,U)$ are TF-ordered and hence $\omega(U,V), \omega(V,U) \in \tes(\W)$,
    using \cref{prop:same_J}.
    We now consider three separate cases.

    Case 1: Suppose that $\W$ is sincere, and so $U,V \not \in \P(\Lambda)$.  
    By \cite[Prop.~3.15]{bhm}, there is at most one left irregular $\tau$-exceptional pair 
    $(X,Y)$  with $\J(X,Y) = \W$. Hence, we can without loss of generality assume that $\omega(V,U)$ is left regular (otherwise we just exchange the roles of $U$ and $V$). Now since $U \notin \P(\Lambda)$ by assumption and $U \notin \Gen V$ by gen-minimality, by Equation~\eqref{eqn:left} we have that 
    \begin{equation*}
        \varphi({\omega(V,U)}) = \varphi{(\Ebm_U V,U)} = (|\Ebm_V(U)|, \Ebm^{-1}_U \Ebm_U V) =
        (\Ebm_V(U), V) =\omega(U,V).
    \end{equation*}

    Case 2: Suppose that $\W$ is not sincere and not a Serre subcategory. Then without loss of generality, we have $V \in \P(\Lambda)$ and $U \notin \P(\Lambda)$. Now since $\W$ is not sincere, $\omega(V,U)$ is left regular; see \cite[Prop.~3.13]{bhm}. As in Case 1, we have $U \notin \Gen V$ by gen-minimality and $U \notin \P(\Lambda)$ by assumption. So, $\varphi(\omega(V,U))$ is computed as in Case 1, and hence we have that $\varphi(\omega(V,U)) = \omega(U,V)$.

\smallskip
    
    Case 3: Suppose that $\W$ is a Serre subcategory, so $U \oplus V \in \P(\Lambda)$. In this case, every $\tau$-exceptional sequence in $\tes(\W)$ is right regular; see \cite[Prop.~3.14]{bhm}. Now for $\ell \in \mathbb{N}$, let $(V_\ell,U_\ell) := \omega^{-1}(\psi^\ell(\omega(V,U)))$.

    We first prove that there is $\ell \in \mathbb{N}$ such that $U_\ell = V$.
    Suppose, for a contradiction, that $U_\ell \neq V$ for all $\ell \in \mathbb{N}$.  We will show by induction on $\ell$ that: 
    \begin{itemize}
    \item[(I1)] $\Gen U_{\ell+1} \subsetneq \Gen U_\ell$,
    \item[(I2)] $V_{\ell+1} = U_\ell$, and
    \item[(I3)] $\Ebm_{U_{\ell+1}}(V_{\ell+1}) \in \P(\J(U_{\ell+1}))$ 
    \end{itemize}
     for all $\ell \in \mathbb{N}$. 

    We first consider the base case $\ell = 0$. In this case, we have $V_0 = V \in \P(\Lambda)$, and so $
        \Ebm_{U_0}(V_0) \in \P(\J(U_0))
    $ by \cref{def:Emap} and Lemma~\ref{lem:proj_bij}. Thus, by 
    Equation~\eqref{eqn:right} we have $$U_1 = |\Ebm_{U_0}^{-1}((\Ebm_{U_0}V_0)[1])|.$$ In particular, either 
    \begin{itemize}
        \item[(i)] $U_1 \in \P(\Lambda)$ and $U_0 \oplus U_1[1]$ is support $\tau$-rigid, or 
        \item[(ii)] $U_1 \in \Gen U_0$ and $U_1 = \Ebm^{-1}_{U_0}((\Ebm_{U_0}V_0)[1])$.
    \end{itemize}
    We consider Case (i) first. 
    We have $\W  = \J(U_1 \oplus V_1)$ by \cref{prop:same_J_mutation},
    and so $\W \subseteq \J(U_1) = U_1^\perp$. Since $\W = (U\oplus V)^\perp$, with $U,V, U_1$ indecomposable projectives, this implies $U_1 \in \add(U \oplus V)$. Since, by assumption, $U_1 \neq V$, we must have $U_1 = U$. But then (i) says that $U \oplus U[1]$ is support $\tau$-rigid, which is absurd. We conclude that we are in Case (ii). Then $U_1 \neq U_0$ since $U_0$ is not in the image of $\Ebm_{U_0}^{-1}$. Since (ii) says that $U_1 \in \Gen U_0$, it follows from Lemma~\ref{lem:proper_containment} that (I1) holds, that is: $\Gen U_1 \subsetneq \Gen U_0$. Then, using Equation~\eqref{eqn:right} for 
    $\psi$ yields 
    \begin{equation*}\label{eq:mut}
        (\Ebm_{U_1} V_1, U_1) = \psi(\Ebm_U V,U) = (\Ebm_{U_1} U ,U_1),
    \end{equation*}
    and so $V_1 = U= U_0$ since $\Ebm_{U_1}$ is a bijection. Thus (I2) holds.
    Finally, since $V_1 \in \P(\Lambda)$ and $U_1 \oplus V_1$ is $\tau$-rigid, also $V_1 \in \P({}^\perp \tau U_1)$. Thus $\Ebm_{U_1}(V_1) \in \P(\J(U_1))$ by Lemma~\ref{lem:proj_bij}, so also (I3) holds. This concludes the proof of the base case.

    Now suppose $\ell > 0$ and that the inductive claim holds for all $\ell' < \ell$. By the induction hypothesis, we have $\Ebm_{U_\ell} V_\ell \in \P(\J(U_\ell))$. As in the base case, the definition of $\psi$ then says that $U_{\ell+1} = |\Ebm_{U_\ell}^{-1}((\Ebm_{U_\ell}V_\ell)[1])|$. In particular, either 
    \begin{itemize}
     \item[(i)]  $U_{\ell+1} \in \P(\Lambda)$ and $U_\ell \oplus U_{\ell+1}[1]$ is support $\tau$-rigid, or
     \item[(ii)] $U_{\ell+1}\in \Gen U_\ell$ and $U_{\ell+1}= \Ebm_{U_\ell}^{-1}((\Ebm_{U_\ell}V_\ell)[1])$.
    \end{itemize}
    In Case (i), we would have $\J(U_{\ell+1} \oplus V_{\ell+1}) = \W$, so 
    $\J(U \oplus V) = \W \subseteq U_{\ell+1}^{\perp}$, with 
    $U,V, U_{\ell+1}$ indecomposable projectives. Since, by assumption,
    $U_{\ell+1} \neq V$, we must have  $U_{/ell+1} =  U$, and hence 
    $\Hom(U, U_{\ell}) =0$, since $U_{\ell} \oplus U_{\ell +1}[1]$ 
    is support $\tau$-rigid. But  
     the induction hypothesis says $U_\ell \in \Gen(U_{\ell-1}) \subseteq \Gen U$, so this is a contradiction.
     
     Thus we are in Case (ii), and moreover $U_{\ell+1} \neq U_\ell$ since $U_\ell$ is not in the image of $\Ebm_{U_\ell}^{-1}$. Since (ii) says that $U_{\ell+1} \in \Gen U_\ell$, it follows from Lemma~\ref{lem:proper_containment} that $\Gen U_{\ell+1} \subsetneq \Gen U_\ell$, so (I1) holds. Then, as in the base case the definition of $\psi$ yields $V_{\ell+1} = U_\ell$ (see \cref{eq:mut}), so (I2) holds.
     
     It remains then to show that also (I3) holds, that is $\Ebm_{U_{\ell+1}}(V_{\ell+1}) \in \P(\J(U_{\ell+1}))$.
Suppose, for a contradiction, that $\Ebm_{U_{\ell+1}}(V_{\ell+1}) \notin \P(\J(U_{\ell+1}))$. Then, by the definition of $\psi$, we have 
$$U_{\ell+2} = \Ebm_{U_{\ell+1}}^{-1}(\Ebm_{U_{\ell+1}}(V_{\ell+1})) = 
V_{\ell+1},$$  and so $U_{\ell+2}= V_{\ell+1} = U_\ell$. Moreover, $$\J(V_{\ell+2} \oplus U_{\ell+2}) = \W = \J(V_{\ell}\oplus U_{\ell})$$ by Propositions~\ref{prop:same_J_mutation} and~\ref{prop:same_J}, so Corollary~\ref{cor:uniqueness}
    implies that also $V_{\ell} = V_{\ell+2}$. Now since $\psi$ and $\omega$ are bijections, the equality $(V_{\ell+2},U_{\ell+2}) = (V_\ell,U_\ell)$ implies the equality $(V_{\ell+1},U_{\ell+1}) = (V_{\ell-1},U_{\ell-1})$. This gives us the contradiction $\Gen U_{\ell-1} = \Gen U_{\ell+1} \subsetneq \Gen U_\ell \subsetneq \Gen U_{\ell-1}$, where the last inclusion is by the induction hypothesis. So (I3) holds also for $\ell$.

The induction is complete, and we have shown that (I1), (I2) and (I3) hold for all $\ell\in\mathbb{N}$.
It follows that there is an infinite chain $\cdots \subsetneq \Gen U_2 \subsetneq \Gen U_1 \subsetneq \Gen U_0$. 
Since $\Lambda$ is $\tau$-tilting finite, there are only finitely many torsion classes in $\mods \Lambda$, all of which are functorially finite, by~\cite[Thm.\ 1.2 and Cor.\ 2.9]{dij}, so we have a contradiction.
We can conclude that there is  $\ell \in \mathbb{N}$ such that $U_\ell = V$.
    We have $\J(V_\ell \oplus U_\ell) = \W$ by Propositions~\ref{prop:same_J_mutation} and~\ref{prop:same_J}. Since $U_\ell = V$, Corollary~\ref{cor:uniqueness} then implies that $V_\ell = U$. 
    We conclude that $\psi^\ell(\omega(V,U)) = \omega(V_\ell,U_\ell) = \omega(U,V)$, or equivalently that $\varphi^{-\ell}(\omega(V,U)) = \omega(U,V)$, as required.
\end{proof}

The following application of the above Lemma generalizes \cite[Lem. 5.1]{bkt}.
      \begin{proposition}\label{prop:TF_transpose} 
      Suppose that we have TF-orderings
    $$\mathcal{Y} = (Y_k,\ldots,Y_{i-1},Y_i,Y_{i+1},\ldots ,Y_n) \text{   and   }
      \mathcal{Y}'= (Y_k,\ldots,Y_{i-1},Y_{i+1},Y_i,Y_{i+2},\ldots,Y_n) $$ 
      of some $\tau$-rigid module $\bigoplus\mathcal{Y}$.
      Then there exists $\ell \in \mathbb{Z}$ such that $\varphi_i^\ell(\omega(\mathcal{Y})) = \omega(\mathcal{Y}').$ 
\end{proposition}

\begin{proof}
Write $\omega(\mathcal{Y}) = (X_k,\ldots,X_n)$.  
Then, by Proposition~\ref{prop:TF_tes_bijection}, we have
$$\omega(\mathcal{Y}') = (X_k,\ldots,X_{i-1},Z_i,Z_{i+1},X_{i+2},\ldots,X_n),$$
for some modules $Z_i$ and $Z_{i+1}$.
For $k\leq j\leq n$, let $M_j=Y_{j+1}\oplus \cdots \oplus Y_n$.
To compute $Z_i, Z_{i+1}$, let $\V= \J(M_{i+1})$.
By Propositions~\ref{lem:E_map_sum} and \ref{prop:TF_tes_bijection} we have 
$$\Ebm_{M_i} = \Ebm_{Y_{i+1} \oplus M_{i+1}} = 
\Ebm^{\J(M_{i+1})}_{\Ebm_{M_{i+1}(Y_{i+1})}} \Ebm_{M_{i+1}}= \Ebm^{\V}_{X_{i+1}}\Ebm_{M_{i+1}}.$$
    It follows that $$Z_{i+1} = \Ebm_{M_{i+1}}(Y_i) = 
    \Ebm_{M_{i+1}}(\Ebm_{M_i})^{-1} (X_i) = \Ebm_{M_{i+1}} \Ebm_{M_{i+1}}^{-1} 
    (\Ebm_{X_{i+1}}^{\V})^{-1} (X_i) =  (\Ebm_{X_{i+1}}^{\V})^{-1} (X_i).$$
Applying Propositions~\ref{lem:E_map_sum} and \ref{prop:TF_tes_bijection} again, we have 
    \begin{align*}
    Z_i = \Ebm_{Y_i \oplus M_{i+1}}(Y_{i+1}) &=
    \Ebm_{Y_i \oplus M_{i+1}}\Ebm^{-1}_{M_{i+1}}(X_{i+1}) = 
    \Ebm^{\J(M_{i+1})}_{\Ebm_{M_{i+1}}(Y_i)} \Ebm_{M_{i+1}}\Ebm^{-1}_{M_{i+1}}
    (X_{i+1})\\ 
   &=\Ebm^{\V}_{Z_{i+1}} \Ebm_{M_{i+1}}\Ebm^{-1}_{M_{i+1}}
    (X_{i+1}) = \Ebm_{Z_{i+1}}^{\V}(X_{i+1}).    
    \end{align*}
    
    Thus, we have $\omega_{\V}(X_{i+1},Z_{i+1}) = (Z_i,Z_{i+1})$ and $\omega_{\V}(Z_{i+1},X_{i+1}) = (X_i,X_{i+1})$. It follows that $X_{i+1} \oplus Z_{i+1}$ is a gen-minimal $\tau_{\V}$-rigid module, since both orderings are TF-orderings. 
    Now Theorem~\ref{thm:gen_min_characterization} says that $X_{i+1} \oplus Z_{i+1} = \Pmods(\V \cap {}^\perp \J_\V(X_{i+1}\oplus Z_{i+1}))$. We therefore obtain the result by applying Lemma~\ref{lem:corank2_1} in the subcategory $\V$ with $\W = \J_\V(X_{i+1}\oplus Z_{i+1})$. 
\end{proof}

As a consequence of Proposition~\ref{prop:TF_transpose} we obtain the following.

\begin{corollary}\label{cor:TF_permutation}
    Let $\W$ be a wide subcategory. Let $r$ be the rank of $\W$,
    and let $\mathcal{X}, \mathcal{X}' \in \tes(\W)$. If $\bigoplus\omega^{-1}(\mathcal{X})$ and $\bigoplus\omega^{-1}(\mathcal{X}')$ are both gen-minimal, then $\mathcal{X}$ and $\mathcal{X}'$ are in the same orbit under $\varphi_{r+1},\psi_{r+1},\ldots,\varphi_{n-1},\psi_{n-1}$.
\end{corollary}
\begin{proof}
    By Corollary~\ref{cor:gen_min_unique}, we have that $\bigoplus \omega^{-1}(\mathcal{X}) = \bigoplus \omega^{-1}(\mathcal{X}')$. Now the fact that this module is gen-minimal means that every ordering of its direct summands is a TF-ordering. Thus the result follows immediately from Proposition~\ref{prop:TF_transpose}.   
\end{proof}

Our approach to transitivity is now to show that every mutation orbit of $\tes(\W)$ contains a $\tau$-exceptional sequence for which the corresponding ordered TF-rigid module has a gen-minimal direct sum. This, combined with Corollary~\ref{cor:TF_permutation}, will give the main result.

\begin{lemma}\label{lem:max_gen_minimal}
    Let $\W$ be a wide subcategory. Let $r$ be the rank of $\W$, and let $\mathcal{O}$ be an orbit of $\tes(\W)$ under $\varphi_{r+1},\psi_{r+1},\ldots,\varphi_{n-1},\psi_{n-1}$. Then there exists $\mathcal{X} \in \mathcal{O}$ such that $\bigoplus \omega^{-1}(\mathcal{X})$ is gen-minimal.
\end{lemma}

\begin{proof}
    Choose $\mathcal{X} \in \mathcal{O}$ such that $\FiltGen(\mathcal{X})$ is maximized, in the sense that we do not have 
    $\FiltGen(\X') \subsetneq \FiltGen(\Y)$ for any $\X'$ in $\mathcal{O}$. Note that such an $\X$ exists because there are finitely many $\tau$-exceptional sequences for $\Lambda$, since $\Lambda$ is assumed to be $\tau$-tilting finite.
    We claim that $\bigoplus \omega^{-1}(\mathcal{X})$ is gen-minimal.
    Assume for a contradiction that it is not.

     Write $\omega^{-1}(\mathcal{X}) = (Y_k,\ldots,Y_n)$ and denote $\T := \FiltGen(\mathcal{X})$. By Lemma~\ref{lem:same_torsion}, we have $$\FiltGen(\X) =
     \Gen (Y_k \oplus\cdots \oplus Y_n),$$ with each $Y_i$ in $\P(\T)$ 
     (by~\cite[Prop.\ 2.9]{air}). 
By the assumption that $\bigoplus \omega^{-1}(\mathcal{X})$ is not gen-minimal, there exists an integer $r$ with $k\leq r\leq n$ such that $Y_r\in \Pmodns(\T)$.
Let $s$ be maximal such that $(Y_k,\ldots Y_{s-1},Y_r,Y_s,\ldots ,\hat{Y}_r,\ldots ,Y_n)$ is not TF-ordered (where the hat indicates omission). Note that, since $(Y_r,Y_k,Y_{k+1},\ldots ,\hat{Y}_r,\ldots ,Y_n)$ is not TF-ordered, such an $s\geq k$ exists, and we have $s\leq r-1$ since $(Y_{r+1},\ldots ,Y_n)$ is TF-ordered. Note also that, by the definition of $s$, we have that $$(Y_k,\ldots ,Y_{s-1},Y_s,Y_r,Y_{s+1},\ldots ,\hat{Y}_r,\ldots ,Y_n)$$ is TF-ordered.

By replacing $\X$ with $\omega(Y_k,\ldots ,Y_{s-1},Y_s,Y_r,Y_{s+1},\ldots ,\hat{Y}_r,\ldots ,Y_n)$
(and relabeling) we see that we can assume that $(Y_{i+1},Y_i,Y_{i+2},\ldots ,Y_n)$ is not TF-ordered for some $i$. By Proposition~\ref{prop:TF_transpose}, we can do this replacement without leaving the orbit $\mathcal{O}$ or changing the torsion class $\T$.

Now let $\mathcal{X}' = \varphi_i(\mathcal{X})$. To obtain a contradiction to the maximality of $\FiltGen(\mathcal{X})$, we will show that $\FiltGen(\mathcal{X}') \supsetneq \FiltGen(\mathcal{X})$. Before showing this, we first relate the terms of $\mathcal{X}$ and $\mathcal{X}'$ more explicitly.
     
Write $\mathcal{X}' = (X'_k,\ldots,X'_n)$
     and $\mathcal{X} = (X_k,\ldots,X_n)$, and let $M = M_{i+1}=Y_{i+2}\oplus \cdots \oplus Y_n$. By \cref{lem:E_map_sum} and the definition of the map $\omega$, we have
    \begin{align}
        \omega^{-1}_{\J(M)}(X_{i},X_{i+1}) &= \left(\left(\Ebm_{X_{i+1}}^{\J(M)}\right)^{-1}(X_{i}),X_{i+1}\right)\nonumber\\
        &=\left(\left(\Ebm_{X_{i+1}}^{\J(M)}\right)^{-1} \circ\Ebm_{M \oplus Y_{i+1}}(Y_{i}),\Ebm_M(Y_{i+1})\right)\nonumber\\
        &=\left(\left(\Ebm_{X_{i+1}}^{\J(M)}\right)^{-1} \circ\Ebm_{X_{i+1}}^{\J(M)} \circ \Ebm_{M}(Y_{i}),\Ebm_M(Y_{i+1})\right)\nonumber \\
        &= (\Ebm_M(Y_i),\Ebm_M(Y_{i+1})).\label{eqn:regular}
    \end{align}
Since $(Y_i,Y_{i+1},Y_{i+2},\ldots ,Y_n)$ is TF-ordered, $Y_i,Y_{i+1}\not\in \Gen M$, and     
since $(Y_{i+1},Y_i,Y_{i+2},\ldots,Y_n)$ is not TF-ordered, we have that $Y_{i+1} \in \Gen(Y_i\oplus M)$. Lemma~\ref{lem:gen_pass_down} (with $X=Y_{i+1}$, $Y=Y_i$ and $Z=M$) thus implies that
     \begin{equation}\label{eqn:gen}X_{i+1} = \Ebm_M(Y_{i+1}) \in \Gen \Ebm_M(Y_i),\end{equation}
     where the equality $X_{i+1} = \Ebm_M(Y_{i+1})$ is from~\eqref{eqn:regular}. This implies that $X_{i+1} \notin \P(\J(M))$ and, by \cite[Lem.~3.12]{bhm}, that $\mathcal{X}$ is left $i$-regular, 
     noting that the computation in \cref{eqn:regular} in particular gives
     that
     \begin{equation}\label{eq:from_comp}
     \left(\Ebm_{X_{i+1}}^{\J(M)}\right)^{-1}(X_{i}) = \Ebm_M(Y_i).    
     \end{equation}
     
     Using the formula for $\varphi_i$ in \cref{def:imutation} (see also
     Equation~\eqref{eqn:left}) yields
     \begin{equation}\label{eqn:Z}X'_{i+1} = \Ebm_M(Y_i).\end{equation} 

     We are now prepared to show that $\FiltGen(\mathcal{X}') \supseteq \FiltGen(\mathcal{X})$. First note that $X'_j = X_j$ for $j \notin \{i,i+1\}$ (since $\varphi_i$ can only change the terms in positions $i$ and $i+1$). Thus it suffices to show that $X_i \oplus X_{i+1} \in \Gen X'_{i+1}$. The fact that $X_{i+1} \in \Gen X'_{i+1}$ follows from~\eqref{eqn:gen} and~\eqref{eqn:Z}. The fact that $X_i \in \Gen X'_{i+1}$ follows from noting that $$X_i = \Ebm_{X_{i+1}}^{\J(M)} (X'_{i+1})$$
     by~\eqref{eq:from_comp} and~\eqref{eqn:Z}(and noting that $X_i$ is a module, so $X_i=f_{X_{i+1}}^{\J(M)}(X'_{i+1})$ by Proposition~\ref{def:Emap}).

    We have shown that $\FiltGen(\mathcal{X}') \supseteq \FiltGen(\mathcal{X})$. Suppose now for a contradiction that $\FiltGen(\mathcal{X}') = \FiltGen(\mathcal{X})$. Then $\Gen(\omega^{-1}(\mathcal{X})) = \FiltGen(\mathcal{X}) = \Gen(\omega^{-1}(\mathcal{X}'))$ by Lemma~\ref{lem:same_torsion} and $\J(\bigoplus \omega^{-1}(\mathcal{X})) = \J(\bigoplus \omega^{-1}(\mathcal{X}'))$ by Propositions~\ref{prop:same_J} and~\ref{prop:same_J_mutation}. By Lemma~\ref{lem:rigid_unique}, this means $\bigoplus \omega^{-1}(\mathcal{X}) = \bigoplus \omega^{-1}(\mathcal{X}')$.
    Denoting $\omega^{-1}(\mathcal{X}') = (W_k,\ldots,W_n)$, it follows that there exists some index $j$ such that $W_j = Y_{i+1}$. On the other hand, note that $W_\ell = Y_\ell$ for $i+1 < \ell$ and that $W_{i+1} = Y_i$ (using $\Ebm_M(Y_i)=X'_{i+1}=\Ebm_M(W_{i+1})$ from Proposition~\ref{prop:TF_tes_bijection} and Equation~\eqref{eqn:Z}). Thus $j \leq i$, but this contradicts the assumption that $(Y_{i+1},Y_i,Y_{i+2},\ldots,Y_n)$ is not TF-ordered. We conclude that $\FiltGen(\mathcal{X}') \supsetneq \FiltGen(\mathcal{X})$, which contradicts that maximality of $\mathcal{X}$. Therefore $\bigoplus \omega^{-1}(\mathcal{X})$ is gen-minimal.
\end{proof}

We are now prepared to prove our main result.

\begin{theorem}\label{thm:transitive_tau_tilting_finite}
    Let $\W$ be a wide subcategory and let $r$ be the rank of $\W$. Then any two $\tau$-exceptional sequences in $\tes(\W)$ are in the same orbit under $\varphi_{r+1},\psi_{r+1},\ldots,\varphi_{n-1},\psi_{n-1}$.
    In particular, letting $\W =0$, any two complete 
    $\tau$-exceptional sequences are in the same orbit under the action of  
    left and right mutations.
\end{theorem}

\begin{proof}
    Let $\mathcal{X}_1, \mathcal{X}_2 \in \tes(\W)$. By Lemma~\ref{lem:max_gen_minimal}, $\mathcal{X}_1$ can be mutated into some $\mathcal{X}'_1 \in \tes(\W)$ with $\bigoplus \omega^{-1}(\mathcal{X}'_1)$ gen-minimal. Likewise $\mathcal{X}_2$ can be mutated into some $\mathcal{X}'_2 \in \tes(\W)$ with $\bigoplus \omega^{-1}(\mathcal{X}'_2)$ gen-minimal. Finally, Corollary~\ref{cor:TF_permutation} implies that $\mathcal{X}'_1$ and $\mathcal{X}'_2$ are related by a sequence of mutations.
\end{proof}

\end{document}